\documentclass[reqno]{amsart}
\usepackage{amsmath}
\usepackage{amssymb}
\usepackage{amsthm}
\usepackage{mathrsfs}
\usepackage{url}
\newtheorem{definition}{Definition}[section]
\newtheorem{theorem}[definition]{Theorem}
\newtheorem{lemma}[definition]{Lemma}
\newtheorem{prop}[definition]{Proposition}

\newtheorem{remark}[definition]{Remark}

\usepackage{color}

\newcommand{\M}{\mathrm{Mult}(H_k)}
\newcommand{\C}{\mathbb{C}}
\newcommand{\D}{\mathbb{D}}
\newcommand{\B}{\mathbb{B}}
\newcommand{\T}{\mathbb{T}}
\newcommand{\bM}{\mathbb{M}}

\newcommand{\A}{\mathcal{A}}
\newcommand{\E}{\mathcal{E}}
\newcommand{\mcH}{\mathcal{H}}
\newcommand{\mf}{\mathfrak{z}}
\newcommand{\mfa}{\mathfrak{a}}
\newcommand{\bb}{\hat{b}(\lambda)}

\begin{document}

\title{Isometric Gelfand transforms of complete Nevanlinna--Pick spaces}
\author{Kenta Kojin}
\address{
Graduate School of Mathematics, Nagoya University, 
Furocho, Chikusaku, Nagoya, 464-8602, Japan
}
\email{m20016y@math.nagoya-u.ac.jp}
\date{\today}

\begin{abstract}
We show that any complete Nevanlinna--Pick space whose multiplier algebra has isometric Gelfand transform (or commutative C*-envelope) is essentially the Hardy space on the open unit disk.
\end{abstract}


 \maketitle


\section{Introduction}


The celebrated Nevenlinna--Pick interpolation theorem \cite{Nev1919, Pic} suggests that the behavior of a multiplier of the Hardy space $H^2$ on $\D=\{z\in\C\;|\;|z|<1\}$ is described by a certain positive semidefinite matrix.
Here, $H^2$ is the reproducing kernel Hilbert space (RKHS for short) of holomorphic functions on $\D$ whose Taylor coefficients at $0$ are square-summable.
Agler \cite{Agl88} saught for more RKHS's on which a Nevanlinna--Pick type interpolation theorem holds. 
In 2000, Agler and McCarthy \cite{AM2000} established a characterization theorem for RKHS's on which the matrix-valued Pick interpolation theorem holds. Such a space is called a complete Nevanlinna--Pick space. The Hardy space is a prototypical example.  They proved that an irreducible RKHS is a complete Nevanlinna--Pick space if and only if it can be embedded into the Drury--Arveson space $H_{\E}$, which can be understood as a 
multivariable Hardy space on the open unit ball $\B_{\E}$ in a Hilbert space $\E$. 
(See Theorem \ref{universality} its the precise statement.) 
Hence, those spaces possess some of the fine structure of $H^2$ (see e.g., \cite{AMp, Hartz2022}). 

In this paper, we always assume that an RKHS $H_k$ has dimension greater than one. The following results motivates us to study irreducible complete Nevanlinna--Pick spaces whose multiplier algebras have isometric Gelfand transform.

\begin{itemize}
\item Clou\^{a}tre and Timko \cite[Section 6]{CT} showed that the Gelfand transform of the multiplier algebra $\M$ of an irreducible complete Nevanlinna--Pick space $H_k$ often fails to be isometric. 

\item Hartz \cite{Har2015} proved that every irreducible complete Nevanlinna--Pick space whose multipliers are hyponormal is isomorphic to $H^2$ as an RKHS.  

\item The author \cite[Theorem 3.9]{Koj} proved that every complete Nevanlinna--Pick space such that the multiplier norm is equal to the uniform norm for all multipliers must be isomorphic to $H^2$ as an RKHS.
\end{itemize}
We prove that $H^2$ is essentially a unique irreducible complete Nevanlinna--Pick space whose multiplier algebra has isometric Gelfand transform. 
We note that the above mentioned Hartz and the author's results immediately follow from this fact.

\begin{theorem}\label{theorem:1-1}
Let $H_k$ be an irreducible complete Nevanlinna--Pick space on a set $X$ normalized at $x_0\in X$. We denote the set of all characters of $\M$ by $\Delta(\M)$. Then the Gelfand transform $g:\M\rightarrow C(\Delta(\M))$ is isometric if and only if there exist a set of uniqueness $A\subset\D$ for $H^2$ and a bijection $j:X\rightarrow A$ such that $j(x_0)=0$ and 
\begin{equation}\label{equation:1-1}
k(x,y)=\frac{1}{1-j(x)\overline{j(y)}}\;\;\;\;(x,y\in X)
\end{equation}
hold. Hence,
\begin{equation*}
H^2\rightarrow H_k,\;\;\;\;f\mapsto f\circ j
\end{equation*}
is a unitary operator. Moreover, in this case, 
\begin{equation*}
H^{\infty}(\D)\rightarrow \M,\;\;\;\;\psi\mapsto\psi\circ j
\end{equation*}
is a unital completely isometric isomorphism. In particular, the Gelfand transform of $\M$ is completely isometric.
\end{theorem}

Clou\^{a}tre and Timko \cite[Theorem 2.4]{CT} proved that the Gelfand transform of a norm-closed unital commutative operator algebra is completely isometric if and only if its C*-envelope is commutative. 
Here, the {\bf C*-envelope} of a unital operator algebra $\A$ is a C*-algebra $C^*_e(\A)$ together with a unital completely isometric homomorphism $\iota:\A\rightarrow C^*_e(\A)$ with the following two properties:
\begin{enumerate}
\item $C^*_e(\A)=C^*(\iota(\A))$
\item whenever $\rho:\A\rightarrow B(H)$ is any other unital completely isometric homomorphism, then there exists a surjective $*$-homomorphism $\pi:C^*(\rho(\A))\rightarrow C^*_e(\A)$ such that $\pi(\rho(a))=\iota(a)$ for all $a\in\A$.
\end{enumerate}
Therefore, $C^*_e(\A)$ is the ``smallest'' C*-algebra generated by a completely isometric copy of $\A$. 

Combining Theorem \ref{theorem:1-1} and \cite[Theorem 2.4]{CT}, we obtain the following result:

\begin{theorem}
    Let $H_k$ be an irreducible complete Nevanlinna--Pick space on a set $X$ normalized at $x_0\in X$. The following statements are equivalent:
    \begin{enumerate}
    \item The Gelfand transform of $\M$ is isometric.
    \item The Gelfand transform of $\M$ is completely isometrtic.
    \item The C*-envelope of $\M$ is commutative.
    \item There exist a set of uniqueness $A\subset\D$ for $H^2$ and a bijection $j:X\rightarrow A$ such that $j(x_0)=0$ and 
    \begin{equation*}
    k(x,y)=\frac{1}{1-j(x)\overline{j(y)}}\;\;\;\;(x,y\in X).
    \end{equation*}
    \end{enumerate}
    Moreover, in this case, 
    \begin{equation*}
    H^2\rightarrow H_k,\;\;\;\;f\mapsto f\circ j
    \end{equation*}
    is a unitary operator and
    \begin{equation*}
    H^{\infty}(\D)\rightarrow \M,\;\;\;\;\psi\mapsto\psi\circ j
    \end{equation*}
    is a unital completely isometric isomorphism.
\end{theorem}

We also discuss the Gelfand transform of a suitable norm-closed subalgebra contained in $\M$ generated by ``coordinate functions'' (see Remark \ref{remark:4-3}). This algebra is a generalization of the disk algebra.


\section{Complete Nevanlinna--Pick spaces}
 Let $H_k$ be an RKHS on a set X with reproducing kernel $k:X\times X\rightarrow\C$. We always assume that $X$ is neither the empty set nor a singleton (in the latter case, $H_k=\C$), and $H_k$ is irreducible in the following sense: For any two points $x,y\in X$, the reproducing kernel $k$ satisfies $k(x,y)\ne 0$, and if $x\ne y$, then $k(\cdot, x)$ and $k(\cdot, y)$ are linearly independent in $H_k$. A function $\phi:X\rightarrow\C$ is called a multiplier of $H_k$ if $\phi f\in H_k$ for all $f\in H_k$. In this case, the multiplication operator $M_{\phi}:H_k\rightarrow H_k$ defined by $M_{\phi}f:=\phi f$ is bounded by the closed graph theorem, and the multiplier norm $\|\phi\|_{\M}$ is defined to be the operator norm of $M_{\phi}$. We know the following useful characterization of contractive matrix-valued multipliers:

 \begin{lemma}[$\mbox{\cite[Theorem 3.1(1)$\Leftrightarrow$(2)]{BTV}}$]\label{lemma:1-1}
 Let $H_k$ be an RKHS on a set $X$. A 
matrix-valued function $\Phi:X\rightarrow\bM_m(\C)$ falls in $\bM_m(\C)\otimes \M$ and satisfies $\|\Phi\|_{\mathbb{M}_m(\C)\otimes\M}\le 1$ if and only if 
 \begin{equation*}
     (I_{\C^m}-\Phi(x)\Phi(y)^*)k(x,y)\;\;\;\;(x,y\in X)
 \end{equation*}
 is positive semidefinite.
 \end{lemma}

 The multiplier algebra of $H_k$, denoted by $\M$, is the algebra consisting of multipliers of $H_k$. It is naturally identified with a subalgebra of $B(H_k)$. 

The following normalization assumption is useful when one studies multiplier algebras: An RKHS $H_k$ is normalized at $x_0\in X$ if $k(x, x_0)=1$ for all $x\in X$. 
Any non-vanishing kernel can be normalized by considering
 \begin{equation*}
    k_{x_0}(x,y):=\frac{k(x,y)k(x_0,x_0)}{k(x,x_0)k(x_0,y)}
 \end{equation*}
 for any fixed $x_0\in X$. This normalization does not change the structure of the multiplier algebra (see \cite[Section 2.6]{AMp}). Therefore, we assume that every irreducible RKHS that we consider below is normalized at some point.


 We recall the definition of complete Nevanlinna--Pick spaces. 

 \begin{definition}\label{def:CNP}
An RKHS $H_k$ on a set $X$ is said to be a {\bf complete Nevanlinna--Pick space} if whenever $m, n\in \mathbb{N}$, $x_1\ldots, x_n$ are points in $X$ and $W_1,\ldots, W_n$ are $m\times m$ matrices such that
\begin{equation*}
    [(I-W_iW_j^*)k(x_i, x_j)]_{i,j=1}^n
\end{equation*}
is positive semidefinite, then there exists a $\Phi\in\mathbb{M}_m(\C)\otimes\M$ with $\Phi(x_i)=W_i$ for $1\le i\le n$ and $\|\Phi\|_{\mathbb{M}_m(\C)\otimes\M}\le 1$.
\end{definition}

Agler and McCarthy \cite{AM2000} established 
the universality of the Drury--Arveson space $H_{\E}^2$. Here, $\E$ is a Hilbert space and $H_{\E}^2$ is the RKHS on $\B_{\E}=\{\zeta\in \E\;|\;\|\zeta\|_{\E}<1\}$ whose reproducing kernel is $\displaystyle k_{\E}(\zeta,\eta)=\frac{1}{1-\langle \zeta,\eta\rangle_{\E}}$ $(\zeta,\eta\in \B_{\E})$.

\begin{theorem}[Agler--McCarthy $\mbox{\cite{AM2000}}$]\label{universality}
Let $H_k$ be an irreducible RKHS on a set $X$ normalized at $x_0\in X$. Then, $H_k$ is a complete Nevanlinna--Pick space if and only if there exist an auxiliary Hilbert space $\E$ and an injection $b:X\rightarrow\B_{\E}$ with $b(x_0)=0$ such that
\begin{equation*}
    k(x,y)=\frac{1}{1-\langle b(x), b(y)\rangle_{\E}}\;\;\;\;(x,y\in X).
\end{equation*}
\end{theorem}

\begin{remark}\label{remarkHartz}
    \upshape In this case, Hartz \cite[Theorem 2.1]{Har2015} proved that the composition operator $C_b:H_{\E}^2\ominus I(A)\rightarrow H_k$ defined by $C_bf:=(f|_A)\circ b$ is a unitary operator, where $A:=b(X)$ and $I(A):=\{f\in H_{\E}^2\;|\;f(z)=0\;(z\in A)\}$.
\end{remark}

In particular, $H^2$ is a complete Nevanlinna--Pick space. Since $H^{\infty}(\D)$ is completely isometrically equal to $\mathrm{Mult}(H^2)$ (see e.g., \cite[Example 4]{CT}), the next result immediately follows from \cite[Theorem 3.1(1)$\Leftrightarrow$(3) and its proof]{BTV}.

\begin{lemma}\label{lemma:1-5}
Let $X$ be a set and $x_0\in X$. Let $j$ be an injection from $X$ into $\D$ such that $j(x_0)=0$. We denote an irreducible complete Nevanlinna--Pick space on $X$ given by reproducing kernel
\begin{equation*}
k(x,y)=\frac{1}{1-j(x)\overline{j(y)}}\;\;\;\;(x,y\in X).
\end{equation*}
by $H_k$.
Then, for any $m\in\mathbb{N}$ and $\Phi\in\bM_m(\C)\otimes \M$ there exists a $\Psi\in \bM_m(\C)\otimes H^{\infty}(\D)$ such that $\Phi=\Psi\circ j$ and $\|\Phi\|_{\bM_m(\C)\otimes\M}=\|\Psi\|_{\infty}$.
\end{lemma}






\section{Gelfand transform of multiplier algebra}

To prove our main result, we use the uniqueness of a solution of a certain extremal problem for two-point Nevanlinna--Pick spaces.

\begin{definition}\label{def:two-point Pick}
We say that an RKHS $H_k$ on a set $X$ satisfies the  {\bf two-point Nevanlinna--Pick property} if whenever $x_1, x_2$ are points in $X$ and $w_1,w_2$ are complex numbers such that
\begin{equation*}
\begin{bmatrix}
    (1-|w_1|^2)k(x_1, x_1)&(1-w_1\overline{w_2})k(x_1, x_2)\\
    (1-w_2\overline{w_1})k(x_2, x_1)&(1-|w_2|^2)k(x_2, x_2)
\end{bmatrix}
\end{equation*}
is positive semidefinite, then there exists a $\phi\in\M$ with $\phi(x_i)=w_i$ for $i=1,2$ and $\|\phi\|_{\M}\le 1$. This matrix is nothing but the $m=1$, $n=2$ case of the matrix in Definition \ref{def:CNP}.
\end{definition}

\begin{lemma}[\mbox{\cite[Proposition 3.1]{Har2017}}]\label{lemma2}
    Let $H_k$ be an irreducible RKHS on a set $X$ normalized at $x_0\in X$. Suppose $H_k$ satisfies the two-point Nevanlinna--Pick property. Then
    \begin{equation*} 
        \sup\{\mathrm{Re}\phi(x)\;|\;\|\phi\|_{\M}\le 1 \;\mathrm{and}\; \phi(x_0)=0\}
        =\left(1-\frac{1}{k(x,x)}\right)^{1/2}
    \end{equation*}
    for any $x\in X$, and this number is strictly positive if $x\ne x_0$. Moreover, there is a unique multiplier $\phi_{x}$ that achieves the supremum if $x\ne x_0$.
\end{lemma}

We need another lemma to prove Theorem \ref{theorem:1-1}. Let $\{e_1, e_2\}$ be an orthonormal system in a Hilbert space $\E$. We denote the norm closure of $\C[\langle \zeta, e_1\rangle, \langle \zeta, e_2\rangle]$ in $\mathrm{Mult}(H_{\E}^2)$ by $\A(e_1, e_2)$. 
Here, in the proof of \cite[Theorem 3.9]{Koj}, we essentially proved that $\langle \zeta,e\rangle_{\E}\;(\zeta\in\B_{\E})$ is a contractive multiplier of $H_{\E}^2$ for every unit vector $e\in \E$.
Moreover, we denote the norm closure of $\C[z_1, z_2]$ in $\mathrm{Mult}(H_{\C^2}^2)$ by $\A_2$.
The following lemma tells us that we can identify $\A(e_1, e_2)$ with $\A_2$:

\begin{lemma}\label{lemma:2-2}
Let $\E$ be a Hilbert space with $\dim{\E}\ge 2$. Let $\{e_1, e_2\}$ be an orthonormal system in $\E$. Then, a unital homomorphism
\begin{equation*}
    \C[\langle \zeta, e_1\rangle, \langle \zeta, e_2\rangle]\rightarrow \C[z_1, z_2],\;\;\;\;p(\langle \zeta, e_1\rangle, \langle \zeta, e_2\rangle)\mapsto p(z_1, z_2)
\end{equation*}
can extend to an isometric isomorphism from $\A(e_1, e_2)$
onto $\A_2$.
\end{lemma}
\begin{proof}
Let $\displaystyle p(z_1, z_2)=\sum_{m_1, m_2}a_{m_1, m_2}z_1^{m_1}z_2^{m_1}$ be a polynomial of indeterminates $z_1$ and $z_2$.
We note that $p(\langle \zeta, e_1\rangle, \langle \zeta, e_2\rangle)$ is a multiplier of $H_{\E}^2$. 
If 
\begin{equation*}
\|p(\langle\zeta, e_1\rangle, \langle \zeta, e_2\rangle)\|_{\mathrm{Mult}(H_{\E}^2)}=1,
\end{equation*}
then Lemma \ref{lemma:1-1} implies that
\begin{equation*}
    \frac{1-p(\langle\zeta, e_1\rangle, \langle \zeta, e_2\rangle)\overline{p(\langle\eta, e_1\rangle, \langle \eta, e_2\rangle)}}{1-\langle \zeta,\eta\rangle}\;\;\;\;(\zeta,\eta\in\B_{\E})
\end{equation*}
is positive semidefinite. 
For arbitrary $(z_1, z_2), (w_1, w_2)\in\B_{\C^2}$, we set $\zeta=z_1e_1+z_2e_2, \eta=w_1e_1+w_2e_2\in\B_{\E}$. Since $\{e_1, e_2\}$ is an orthonormal system,
we have
\begin{align*}
&\frac{1-p(z_1, z_2)\overline{p(w_1, w_2)}}{1-\langle (z_1, z_2), (w_1, w_2)\rangle_{\C^2}}\\
&=\frac{1-\left(\displaystyle\sum_{m_1, m_2}a_{m_1, m_2}z_1^{m_1}z_2^{m_2}\right)\overline{\left(\displaystyle\sum_{m_1, m_2}a_{m_1, m_2}w_1^{m_1}w_2^{m_2}\right)}}{1-(z_1\overline{w_1}+z_2\overline{w_2})}\\
&=\frac{1-\left(\displaystyle\sum_{m_1, m_2}a_{m_1, m_2}\langle \zeta, e_1\rangle^{m_1}\langle \zeta, e_2\rangle^{m_2}\right)\overline{\left(\displaystyle\sum_{m_1, m_2}a_{m_1, m_2}\langle \eta, e_1\rangle^{m_1}\langle \eta, e_2\rangle^{m_2}\right)}}{1-\langle \zeta, \eta\rangle_{\E}}\\
&=\frac{1-p(\langle\zeta, e_1\rangle, \langle \zeta, e_2\rangle)\overline{p(\langle \eta, e_1\rangle, \langle \eta, e_2\rangle)}}{1-\langle \zeta, \eta\rangle_{\E}}.
\end{align*}
Thus, 
\begin{equation*}
    \frac{1-p(z_1, z_2)\overline{p(w_1, w_2)}}{1-\langle (z_1, z_2), (w_1, w_2)\rangle_{\C^2}}\;\;\;\;((z_1, z_2), (w_1, w_2)\in\B_{\C^2})
\end{equation*}
is positive semidefinite, and hence Lemma \ref{lemma:1-1} tells us that $p(z_1, z_2)$ is a contractive multiplier of $H_{\C^2}^2$. Therefore, we have 
\begin{equation*}
\|p(z_1, z_2)\|_{\mathrm{Mult}(H_{\C^2}^2)}\le\|p(\langle \zeta, e_1\rangle, \langle \zeta, e_2\rangle)\|_{\mathrm{Mult}(H_{\E}^2)} 
\end{equation*}
for all $p\in\C[z_1, z_2]$.

We prove the opposite inequality. Let $p\in\C[z_1, z_2]$ be a polynomial with $\|p(z_1, z_2)\|_{\mathrm{Mult}(H_{\C^2}^2)}=1$.
For all $\zeta, \eta\in \B_{\E}$, we have
\begin{align}\label{equation:2-1}
&\frac{1-p(\langle\zeta, e_1\rangle, \langle \zeta, e_2\rangle)\overline{p(\langle \eta, e_1\rangle, \langle \eta, e_2\rangle)}}{1-\langle \zeta, \eta\rangle} \notag \\
&=\frac{1-p(\langle\zeta, e_1\rangle, \langle \zeta, e_2\rangle)\overline{p(\langle \eta, e_1\rangle, \langle \eta, e_2\rangle)}}{1-\langle \zeta, e_1\rangle\langle e_1, \eta\rangle-\langle \zeta, e_2\rangle\langle e_2, \eta\rangle}\times
\frac{1-\langle \zeta, e_1\rangle\langle e_1, \eta\rangle-\langle \zeta, e_2\rangle\langle e_2, \eta\rangle}{1-\langle \zeta, \eta\rangle}.
\end{align}
Let $\{e_i\}_{i\in I}$ be an orthonormal basis for $\E$ containing $e_1$ and $e_2$. As an application of Parseval's identity (see e.g., \cite[Theorem 4.18]{Rud}), we get
\begin{align}\label{equation:3-3}
\frac{1-\langle \zeta, e_1\rangle\langle e_1, \eta\rangle-\langle \zeta, e_2\rangle\langle e_2, \eta\rangle}{1-\langle\zeta, \eta\rangle}
&=\frac{1-\displaystyle\sum_{i\in I}\langle \zeta, e_i\rangle \langle e_i, \eta\rangle+\sum_{i\in I\setminus\{1,2\}}\langle \zeta, e_i\rangle\langle e_i, \eta\rangle}{1-\langle \zeta, \eta\rangle} \notag \\
&=1+\frac{\displaystyle\sum_{i\in I\setminus\{1,2\}}\langle\zeta, e_i\rangle \langle e_i, \eta\rangle}{1-\langle \zeta, \eta\rangle}.
\end{align}
Thus, the Schur product theorem shows that
\begin{equation*}
\frac{1-\langle \zeta, e_1\rangle\langle e_1, \eta\rangle-\langle \zeta, e_2\rangle\langle e_2, \eta\rangle}{1-\langle\zeta, \eta\rangle}
\end{equation*}
is positive semidefinite. Since $\|p(z_1, z_2)\|_{\mathrm{Mult}(H_{\C^2}^2)}=1$, Lemma \ref{lemma:1-1} implies that 
\begin{equation*}
\frac{1-p(\langle\zeta, e_1\rangle, \langle \zeta, e_2\rangle)\overline{p(\langle \eta, e_1\rangle, \langle \eta, e_2\rangle)}}{1-\langle \zeta, e_1\rangle\langle e_1, \eta\rangle-\langle \zeta, e_2\rangle\langle e_2, \eta\rangle}\;\;\;\;(\zeta, \eta\in\B_{\E})
\end{equation*}
is positive semidefinite, and hence the Schur product theorem and equation (\ref{equation:2-1}) tells us that
\begin{equation*}
    \frac{1-p(\langle\zeta, e_1\rangle, \langle \zeta, e_2\rangle)\overline{p(\langle \eta, e_1\rangle, \langle \eta, e_2\rangle)}}{1-\langle \zeta, \eta\rangle}\;\;\;\;(\zeta, \eta\in \B_{\E})
\end{equation*}
is positive semidefinite. By Lemma \ref{lemma:1-1} again, $\|p(\langle \zeta, e_1\rangle, \langle \zeta, e_2\rangle)\|_{\mathrm{Mult}(H_{\E}^2)}\le 1$.

In conclusion, 
\begin{equation*}
\|p(z_1, z_2)\|_{\mathrm{Mult}(H_{\C^2}^2)}=\|p(\langle \zeta, e_1\rangle, \langle\zeta, e_2\rangle)\|_{\mathrm{Mult}(H_{\E}^2)}
\end{equation*}
holds for every $p\in\C[z_1, z_2]$. Therefore, the unital homomorphism
\begin{equation*}
    \C[\langle \zeta, e_1\rangle, \langle \zeta, e_2\rangle]\rightarrow \C[z_1, z_2],\;\;\;\;p(\langle \zeta, e_1\rangle, \langle \zeta, e_2\rangle)\mapsto p(z_1, z_2)
\end{equation*}
can extend to an isometric isomorphism from $\A(e_1, e_2)$ onto $\A_2$.
\end{proof}

A subset $A\subset\D$ is said to be a {\bf set of uniqueness for $H^2$} if the only element of $H^2$ that vanishes on $A$ is the zero function.
Since the zeros of a nonzero function in $H^2$ satisfy the Blaschke condition \cite[Theorem 15.23]{Rud}, $A\subset\D$ is a set of uniqueness for $H^2$ if and only if
\begin{equation*}
\sum_{a\in A}(1-|a|)=\infty.
\end{equation*}


\begin{proof}[Proof of Theorem \ref{theorem:1-1}]
If $H_k$ is an irreducible complete Nevanlinna--Pick space normalized at $x_0$, then Theorem \ref{universality} shows that there exist a Hilbert space $\E$ and an injection $b:X\rightarrow\B_{\E}$ with $b(x_0)=0$ such that
\begin{equation*}
    k(x,y)=\frac{1}{1-\langle b(x), b(y)\rangle_{\E}}\;\;\;\;(x,y\in X).
\end{equation*}

Let $\{e_1, e_2\}$ be an orthonormal system in $\E$. 
We prove that 
\begin{equation}\label{equation:1-1-1}
\langle b(x), e_1\rangle_{\E}+\frac{1}{2}\langle b(x), e_2\rangle_{\E}^2\;\;\;\;(x\in X) 
\end{equation}
is a contractive multiplier of $H_k$. We know that this function is a multiplier of $H_k$ (see the proof of \cite[Theorem 3.9]{Koj}).
Since the Gelfand transform $g:\M\rightarrow C(\Delta (\M))$ is isometric, we have
\begin{align*}
&\left\|\langle b(\cdot), e_1\rangle+\frac{1}{2}\langle b(\cdot), e_2\rangle^2\right\|_{\M}\\
&=\left\|g(\langle b(\cdot), e_1\rangle+\frac{1}{2}\langle b(\cdot), e_2\rangle^2)\right\|_{\infty}\\
&=\sup\left\{|\chi(\langle b(\cdot), e_1\rangle+\frac{1}{2}\langle b(\cdot), e_2\rangle^2)|\;\middle|\;\chi\in\Delta(\M)\right\}.
\end{align*}
Since an operator 
$C_b:\mathrm{Mult}(H_{\E}^2)\rightarrow \M$ defined by $C_b\phi\mapsto\phi\circ b$
is a unital homomorphism, $\chi\circ C_b$ is a character of $\mathrm{Mult}(H_{\E}^2)$ for every $\chi\in\Delta(\M)$. 
Therefore, we obtain
\begin{align*}
&\sup\left\{|\chi(\langle b(\cdot), e_1\rangle+\frac{1}{2}\langle b(\cdot), e_2\rangle^2)|\;\middle|\;\chi\in\Delta(\M)\right\}\\
&\le\sup\left\{|\pi(\langle \zeta, e_1\rangle+\frac{1}{2}\langle \zeta, e_2\rangle^2)|\;\middle|\;\pi\in\Delta(\mathrm{Mult}(H_{\E}^2))\right\}.
\end{align*}
As $\displaystyle\langle \zeta, e_1\rangle+\frac{1}{2}\langle \zeta, e_2\rangle^2\in\A(e_1, e_2)$ and $\A(e_1, e_2)\subset \mathrm{Mult}(H_{\E}^2)$, we get
\begin{align*}
     &\sup\left\{|\pi(\langle \zeta, e_1\rangle+\frac{1}{2}\langle \zeta, e_2\rangle^2)|\;\middle|\;\pi\in\Delta(\mathrm{Mult}(H_{\E}^2))\right\}\\
    &\le \sup\left\{|\rho(\langle \zeta, e_1\rangle+\frac{1}{2}\langle \zeta, e_2\rangle^2)|\;\middle|\;\rho\in\Delta(\A(e_1, e_2))\right\}
\end{align*}
By Lemma \ref{lemma:2-2}, we obtain 
\begin{align*}
    \sup\left\{|\rho(\langle \zeta, e_1\rangle+\frac{1}{2}\langle \zeta, e_2\rangle^2)|\;\middle|\;\rho\in\Delta(\A(e_1, e_2))\right\}
    =\sup\{|\rho(z_1+\frac{1}{2}z_2^2)|\;|\;\rho\in\Delta(\A_2)\}.
\end{align*}
Since $\Delta(\A_2)$ is (homeomorphic to) $\overline{\B_2}$, we have
\begin{align*}
&\sup\{|\rho(z_1+\frac{1}{2}z_2^2)|\;|\;\rho\in\Delta(\A_2)\}\\
&=\sup_{(z_1, z_2)\in\overline{\B_2}}|z_1+\frac{1}{2}z_2^2|\\
&\le \sup_{(z_1, z_2)\in\overline{\B_2}}|z_1|+\frac{1}{2}|z_2|^2\\
&\le \sup_{z_1\in\overline{\D}}|z_1|+\frac{1}{2}(1-|z_1|^2)\;\;(\mbox{because $|z_1|^2+|z_2|^2\le 1$}).
\end{align*}
A tedious calculation enables us to prove that a function 
\begin{equation*}
g(t):=t+\frac{1}{2}(1-t^2)
\end{equation*}
is monotonically increasing on the interval $[0,1]$ with $g(1)=1$. Thus, we obtain
\begin{equation*}
\sup_{t\in [0,1]}g(t)=1.
\end{equation*}
Therefore, we have 
\begin{equation*}
\sup_{z_1\in\overline{\D}}|z_1|+\frac{1}{2}(1-|z_1|^2)\le 1,
\end{equation*}
and hence 
\begin{equation*}
\left\|\langle b(\cdot), e_1\rangle+\frac{1}{2}\langle b(\cdot), e_2\rangle^2\right\|_{\M}\le 1.
\end{equation*}

Next, we prove that there exists an injection $j:X\rightarrow \D$ such that $j(x_0)=0$ and
\begin{equation*}
    k(x,y)=\frac{1}{1-j(x)\overline{j(y)}}\;\;\;\;(x,y\in X)
\end{equation*}
holds. We use the same scheme as the proof of \cite[Theorem 3.9]{Koj}. 
We fix a point $\lambda\in X\setminus \{x_0\}$ and set $\displaystyle \bb:=\frac{b(\lambda)}{\|b(\lambda)\|}$ (n.b., $b(\lambda)\ne 0$ because $b$ is injective and $b(x_0)=0$). Then, $\langle b(\cdot), \bb\rangle$ is a contractive multiplier of $H_k$ such that $\langle b(x_0), \bb\rangle=0$ and 
\begin{equation*}
\langle b(\lambda),\bb\rangle=\|b(\lambda)\|=\left(1-\frac{1}{k(\lambda, \lambda)}\right)^{1/2}. 
\end{equation*}
Let $\{e_i\}_{i\in I}$ be an orthonormal basis for $\E$ containing $\bb$. For any $i\in I$ with $e_i\ne \bb$, we define a multiplier of $H_k$ by
\begin{equation*}
\langle b(\cdot),\bb\rangle+\frac{1}{2}\langle b(\cdot), e_i\rangle^2.
\end{equation*}
    Obviously, $\displaystyle\langle b(x_0), \bb\rangle+\frac{1}{2}\langle b(x_0),e_i\rangle^2=0$. Moreover, we have
\begin{equation*}
\langle b(\lambda),\bb\rangle+\frac{1}{2}\langle b(\lambda), e_i\rangle^2=\langle b(\lambda), \bb\rangle=\|b(\lambda)\|=\displaystyle\left(1-\frac{1}{k(\lambda,\lambda)}\right)^{1/2}
\end{equation*}
because $\bb$ and $e_i$ are orthogonal. We have already seen that
\begin{equation*}
\left\|\langle b(\cdot), \bb\rangle+\frac{1}{2}\langle b(\cdot), e_i\rangle^2\right\|_{\M}\le 1.
\end{equation*}
Therefore, the uniqueness part of Lemma \ref{lemma2} implies 
\begin{equation*}
\langle b(\cdot),\bb\rangle+\frac{1}{2}\langle b(\cdot), e_i\rangle^2=\langle b(\cdot),\bb\rangle
\end{equation*}
for every $i\in I$ with $e_i\ne \bb$. So, we conclude that $\langle b(x), e_i\rangle= 0$ for all $x\in X$ if $e_i\ne \bb$. Thus, we have
\begin{equation}\label{equation2}
    b(x)=\sum_{i\in I}\langle b(x), e_i\rangle e_i=\langle b(x), \bb\rangle \bb
\end{equation}
for all $x\in X$. We define $j:X\rightarrow \D$ by $j(x):=\langle b(x), \bb\rangle$. Since $b(x_0)=0$, we get $j(x_0)=0$.
The injectivity of $b$ and equation (\ref{equation2}) imply that $j$ must be injective.
Moreover, using equation (\ref{equation2}) we can prove 
\begin{equation*}
    k(x,y)=\frac{1}{1-j(x)\overline{j(y)}}\;\;\;\;(x,y\in X).
\end{equation*}

It remains to prove that $A:=j(X)\subset\D$ is a set of uniqueness for $H^2$.
On the contrary, suppose that $A$ is not. Then, we have
\begin{equation*}
\sum_{x\in X} (1-|j(x)|)<\infty.
\end{equation*}
Therefore, \cite[Theorem 15.21]{Rud} tells us that
\begin{equation*}
\theta(z):=\prod_{x\in X}\frac{j(x)-z}{1-\overline{j(x)}z}\frac{|j(x)|}{j(x)}\;\;\;\; (z\in\D)
\end{equation*}
converges and defines an inner function. 
We define a unital homomorphism $\pi: H^{\infty}(\D)\rightarrow \M$ by 
\begin{equation*}
\pi(\psi)=\psi\circ j\;\;\;\;(\psi\in H^{\infty}(\D)).
\end{equation*}
By \cite[Theorem 17.9]{Rud}, we have $\ker\pi=\theta H^{\infty}(\D)$.  
Since $H^{\infty}(\D)$ is isometrically equal to $\mathrm{Mult}(H^2)$, Lemma \ref{lemma:1-1} implies that $\pi$ is contractive. 
In fact, we have
\begin{equation*}
(1-\pi(\psi(x))\overline{\pi(\psi(y))})k(x,y)=\frac{1-\psi(j(x))\overline{\psi(j(y))}}{1-j(x)\overline{j(y)}}.
\end{equation*}
Moreover, Lemma \ref{lemma:1-5} implies that we can show that $H^{\infty}(\D)/\theta H^{\infty}(\D)$ is isometrically isomorphic to $\M$. Since the Gelfand transform of $\M$ is isometric, \cite[Theorem 4.2 (ii)]{Clo2015} shows that $\theta$ is an automorphism of $\D$, and hence $A$ is a singleton. This is a contradiction. Therefore, $A$ is a set of uniqueness for $H^2$.

Conversely, with a reproducing kernel $k$ given by equation (\ref{equation:1-1}), $H_k$ is a normalized irreducible complete Nevanlinna--Pick space. 
Since $H^{\infty}(\D)$ is completely isometrically equal to $\mathrm{Mult}(H^2)$ (see e.g., \cite[Example 4]{CT}), Lemma \ref{lemma:1-1} implies that a unital homomorphism
$\pi: H^{\infty}(\D)\rightarrow\M$ defined by
\begin{equation*}
\pi(\psi)=\psi\circ j\;\;\;\;(\psi\in H^{\infty}(\D))
\end{equation*}
is completely contractive. 
As an application of Lemma \ref{lemma:1-5}, we can show that $H^{\infty}(\D)/\ker \pi$ is completely isometrically isomorphic to $\M$. 
Since $j(X)$ is a set of uniqueness for $H^2$, $\ker\pi=\{0\}$. 
Hence, $H^{\infty}(\D)$ is completely isometrically isomorphic to $\M$. It is well known that the Gelfand transform of $H^{\infty}(\D)$ is completely isometric (see e.g., \cite[Example 4]{CT}). So we are done. 
\end{proof}

\begin{remark}\label{remark:Hartz}
\upshape Our idea of proving that equation (\ref{equation:1-1-1}) defines a contractive multiplier is to reduce the problem to a function algebra situation whose character space is well understood. 
However, Professor Michael Hartz gave us an idea to prove it directly without using such a reduction procedure.
We will give a sketch of his idea as follows.
A similar argument of equation (\ref{equation:3-3}) enables us to prove that $(\langle b(\cdot), e_1\rangle, \langle b(\cdot), e_2\rangle)$ is a row contraction on $H_k$.
Since all the characters of $\M$ are completely contractive, we have
\begin{equation*}
|\chi(\langle b(\cdot), e_1\rangle)|^2+|\chi(\langle b(\cdot), e_2\rangle)|^2\le 1
\end{equation*}
for every $\chi\in\Delta(\M)$. 
Hence, we obtain
\begin{align*}
|\chi(\langle b(\cdot), e_1\rangle + \frac{1}{2}\langle b(\cdot), e_2\rangle^2)|
&\le |\chi(\langle b(\cdot), e_1\rangle)|+\frac{1}{2}(1-|\chi(\langle b(\cdot), e_1\rangle)|^2)\le 1
\end{align*}
for every $\chi\in\Delta(\M)$ by a similar argument of the proof of Theorem \ref{theorem:1-1}.
Since the Gelfand transform of $\M$ is isometric, $\displaystyle \langle b(\cdot), e_1\rangle_{\E} + \frac{1}{2}\langle b(\cdot), e_2\rangle_{\E}^2$ is contractive.
\end{remark}

The following proposition tells us that an injection $j$ in Theorem \ref{theorem:1-1} is unique up to rotation:

\begin{prop}\label{proposition:2-5}
Let $X$ be a set with $x_0\in X$.
Let $j$ and $j'$ be injections from $X$ into $\D$ such that $j(x_0)=j'(x_0)=0$ and 
\begin{equation*}
\frac{1}{1-j(x)\overline{j(y)}}=\frac{1}{{1-j'(x)\overline{j'(y)}}}\;\;\;\;(x,y\in X).
\end{equation*}
Then, there exists an $\alpha\in\T=\{z\in\C\;|\;|z|=1\}$ such that
\begin{equation*}
j'(x)=\alpha j(x)
\end{equation*}
holds for every $x\in X$.
\end{prop}

\begin{proof}
Obviously, $j(x)\overline{j(y)}=j'(x)\overline{j'(y)}$ holds for every $x,y\in X$. In particular, we have $|j(x)|=|j'(x)|$ for all $x\in X$. 
Thus, every $x\in X$ admits an $\alpha_x\in\T$ such that $j'(x)=\alpha_x j(x)$. Therefore, we have
\begin{equation*}
j(x)\overline{j(y)}=j'(x)\overline{j'(y)}=\alpha_x\overline{\alpha_y}j(x)\overline{j(y)}.
\end{equation*}
Since $j$ is injective and $j(x_0)=0$, we get $j(x)\ne 0$ for all $x\in X\setminus\{x_0\}$, and hence 
\begin{equation*}
\alpha_x=\alpha_y
\end{equation*}
holds for every $x,y\in X\setminus\{x_0\}$. As $j$ and $j'$ send $x_0$ to $0$, we complete the proof.
\end{proof}

\begin{remark}\label{remark:4-3}
\upshape Let $H_k$ be a complete Nevanlinna--Pick space on a set $X$ normalized at $x_0\in X$. Theorem \ref{universality} tells us that there exist a Hilbert space $\E$ and an injection $b:X\rightarrow\B_{\E}$ with $b(x_0)=0$ such that
\begin{equation*}
    k(x,y)=\frac{1}{1-\langle b(x), b(y)\rangle_{\E}}\;\;\;\;(x,y\in X).
\end{equation*}
Let $E=\{e_i\}_{i\in I}$ be an orthonormal basis for $\E$. Let $\A(E)\subset\mathrm{Mult}(H_{\E}^2)$ denote a norm closed unital operator algebra generated by $\{\langle b(\cdot), e_i\rangle_{\E}\;|\;i\in I\}$.
This is a generalization of the disk algebra $A(\D)$. 
Since $\displaystyle\langle b(\cdot), e_i\rangle+\frac{1}{2}\langle b(\cdot), e_j\rangle^2\;(i\ne j)$ falls in $\A(E)$, a part of the proof of Theorem \ref{theorem:1-1} still works if we replace $\M$ with $\A(E)$. In fact, we can prove that whenever the Gelfand transform of $\A(E)$ is isometric, there exists an injection $j:X\rightarrow \D$ such that $j(x_0)=0$ and 
\begin{equation*}
k(x,y)=\frac{1}{1-j(x)\overline{j(y)}}\;\;\;\;(x,y\in X).
\end{equation*}
However, $j(X)\subset\D$ does not become a set of uniqueness for $H^2$ in general even if the Gelfand transform of $\A(E)$ is completely isometric. 
We consider Arveson's example in \cite[page 204]{Arv} to see this fact. Let $\{\lambda_n\}_{n=1}^{\infty}$ be a countable dense subset of $\T$, and set $\mf_n=(1-n^{-2})\lambda_n\in \D$.
The Blaschke condition $\displaystyle\sum_{n=1}^{\infty}(1-|\mf_n|)<\infty$ yields the Blaschke product $\theta$ having a simple zero at each $\mf_n$ \cite[Theorem 15.21]{Rud}. 
Let $H_k$ be a complete Nevanlinna--Pick space on $X=\{\mf_n\}_{n=1}^{\infty}\subset\D$ whose reproducing kernel is
\begin{equation*}
k(z,w)=\frac{1}{1-z\overline{w}}\;\;\;\;(z,w\in X).
\end{equation*}
Then, \cite[Corollary 5.8]{PR} implies that
\begin{equation*}
U:H^2\ominus I(X)\rightarrow H_k,\;\;\;\;f\mapsto f|_X
\end{equation*}
is a unitary operator, where $I(X)=\{f\in H^2\;|\;f(z)=0\;(z\in X)\}$. We define a weak-$*$ closed ideal $\mfa\subset \M$ by $\mfa=\theta H^{\infty}(\D)$. The identity $\mfa H^2=I(X)$ immediately follows from \cite[Theorem 17.9]{Rud}. We set $\mcH_{\mfa}=H^2\ominus\mfa H^2$, which is a coinvariant subspace of $H^{\infty}(\D)$. It is not hard to check that
\begin{equation*}
U(P_{\mcH_{\mfa}}M_{\psi}P_{\mcH_{\mfa}})U^*=M_{\psi|_X}
\end{equation*}
holds for every $\psi\in H^{\infty}(\D)$. Here, $P_{\mcH_{\mfa}}:H^2\rightarrow \mcH_{\mfa}$ is the orthogonal projection into $\mcH_{\mfa}$. 
Moreover, Lemma \ref{lemma:1-5} implies that every $\Phi\in\bM_m(\C)\otimes \M$ admits a $\Psi\in\bM_m(\C)\otimes H^{\infty}(\D)$ such that $\Psi|_X=\Phi$ and $\|\Phi\|_{\bM_m(\C)\otimes\M}=\|\Psi\|_{\infty}$.
Therefore, $\M$ is completely isometrically isomorphic to $\mathcal{M}_{\mfa}=\{P_{\mcH_{\mfa}}M_{\psi}P_{\mcH_{\mfa}}\;|\;\psi\in H^{\infty}(\D)\}$. 
In particular, the norm closure of $\C[z]$ in $\M$, denoted by $\A$, is completely isometrically isomorphic to 
\begin{equation*}
\A_{\mfa}=\overline{\{P_{\mcH_{\mfa}}M_{q}P_{\mcH_{\mfa}}\;|\;q\in A(\D)\}}^{\mathrm{norm}}=\overline{\{P_{\mcH_{\mfa}}M_{q}P_{\mcH_{\mfa}}\;|\;q\in \C[z]\subset A(\D)\}}^{\mathrm{norm}}. 
\end{equation*}
Let $K\subset\overline{\D}$ be the closure of the zero set of $\theta$, along with the points on $\T$ contained in an arc across which $\theta$ cannot be continued holomorphically. The construction of $\theta$ implies $K\cap \T=\T$ (cf. \cite[page 204]{Arv}). 
Thus, the Gelfand transform of $\A$ is completely isometric as mentioned in \cite[Example 4]{CT}. 
However, $X$ is not a set of uniqueness for $H^2$ because $X=\{\mf_n\}_{n=1}^{\infty}$ satisfies the Blaschke condition $\displaystyle\sum_{n=1}^{\infty}(1-|\mf_n|)<\infty$. 
If $j:X\rightarrow \D$ is an injection such that $j(x_0)=0$ and
\begin{equation*}
k(z,w)=\frac{1}{1-j(z)\overline{j(w)}}\;\;\;\;(z,w\in X),
\end{equation*}
then Proposition \ref{proposition:2-5} implies that $j$ is rotation, and hence $j(X)$ is not a set of uniqueness for $H^2$.
\end{remark}

\section*{acknowledgement}
The author acknowledges Professor Michael Hartz for his comments to a draft version of this paper. In particular, Remark \ref{remark:Hartz} is due to his idea. The author also acknowledges Professor Narutaka Ozawa for inviting him to Kyoto University, where some part of this work was done.
The author is grateful to Professor Yoshimichi Ueda for his comments on the draft of this paper. This work was supported by JSPS Research Fellowship for Young Scientists (KAKENHI Grant Number JP 23KJ1070).

\end{document}